\documentclass{amsart}
\usepackage{amsmath}%
\usepackage{amsfonts}%
\usepackage{amssymb}%
\usepackage{graphicx}
\usepackage{color}
\newtheorem{theorem}{Theorem}
\theoremstyle{plain}

\newtheorem{lemma}{Lemma}

\numberwithin{equation}{section}
\begin{document}
\title[iterates on escaping Fatou components]{iterates of meromorphic functions on escaping Fatou components}
\author[Zheng and Wu]{Zheng Jian-Hua and Wu Cheng-Fa$^*$}
\keywords{Escaping Sets, Meromorphic Functions, Fatou set, Julia set. \\  {\color{white} aa} $^*$Corresponding author}%
\address{Department of Mathematical Sciences, Tsinghua University, Beijing, 100084, P. R. China}
\email{zheng-jh@mail.tsinghua.edu.cn}
\address{Institute for Advanced Study, Shenzhen University, Shenzhen, 518060, P. R. China}
\address{College of Mathematics and Statistics, Shenzhen University, Shenzhen, 518060, P. R. China}
\email{cfwu@szu.edu.cn}

\begin{abstract}In this paper, we prove that the ratio of the modulus of the iterates of two points in an escaping Fatou component could be bounded even if the orbit of the component contains a sequence of annuli whose moduli tend to infinity, and this cannot happen when the maximal modulus of the meromorphic function is uniformly large enough. In this way we extend certain related results for entire functions to meromorphic functions with infinitely many poles.
\end{abstract}

\maketitle




\section{Introduction and Main Results}

Let $f$ be a meromorphic function that is not a M\"{o}bius
transformation and let $f^{n}$, $n\in \mathbb{N}$, denote the $n$th
iterate of $f$. The Fatou set $F(f)$ of $f$ is defined as the set of
points, $z\in \hat{\mathbb{C}}$, such that
$\{f^{n}\}_{n\in\mathbb{N}}$ is well-defined and forms a normal
family in some neighborhood of $z$. Therefore, $F(f)$ is an open
set. Since $F(f)$ is completely invariant under $f$, i.e., $f(z)\in
F(f)$ if and only if $z\in F(f)$, for a component $U$ of $F(f)$,
$f(U)$ is contained in certain component of $F(f)$. We write $U_n$ for the
component of $F(f)$ such that $f^n(U)\subseteq U_n$. The complement
$J(f)=\hat{\mathbb{C}}\setminus F(f)$ of $F(f)$ is called the Julia set
of $f$. Then $\bigcup_{n=0}^\infty f^{-n}(\infty)\subset J(f)$ and
if $\bigcup_{n=0}^\infty f^{-n}(\infty)$ contains three distinct
points in the extended complex plane $\hat{\mathbb{C}}$, we have
$\overline{\bigcup_{n=0}^\infty f^{-n}(\infty)}=J(f)$. Thus every
$f^n$ is analytic in $F(f)$ for $n\geq 1$.

In this paper, we investigate Fatou components $U$ of $f$ with
$f^n|_U\to\infty \, (n\to\infty)$. For convenience, we call such $U$ an escaping Fatou component of $f$. We first
determine what kinds of Fatou components will be escaping to
$\infty$ under iterates.

A periodic Fatou component $W$ of $f$ with period $p\geq 1$, i.e.,
$f^p(W)\subseteq W=W_p$ and $p$ is the minimal positive integer such
that the inclusion holds, is known as a Baker domain if
$f^{np}|_W\to a \, (n\to\infty)$, but $f^p$ is not defined at $a$. Then
$\infty$ must be the limit value of $f^{np+j}$ in $W$ for some
$0\leq j\leq p-1$ and so $W_j$ is unbounded. If an escaping Fatou
component $U$ of $f$ is periodic, i.e., for some positive integer
$p$, $f^p(U)\subseteq U=U_p$, then $U$ is a Baker domain of $f$. The
periodic Fatou components can be classified into five possible types:
attracting domains, parabolic domains, Siegel disks, Herman rings
and Baker domains (see \cite{ber}, Theorem 6).

A Fatou component $U$ is called a wandering domain of $f$, if for
any pair of positive integers $m\not= n$, $U_m\not=U_n$ and actually
$U_m\cap U_n=\emptyset$. We know that the escaping Fatou component
$U$ is either a wandering domain, a Baker domain or the preimage of a
Baker domain. Many examples of such Baker domains and wandering
domains have been revealed; see \cite{R} for a survey about Baker
domains and \cite{Baker1, Baker2, BKL1, Her, EL, KS,
RipponStallard2, Ber, S, Bi, BZ, FGJ, L} for studies on wandering domains. In addition, the existence of wandering
domains with special geometric or dynamical behaviors
and/or for special classes of meromorphic functions continues to
be an important topic and to attract much interest.

Let $U$ be an escaping Fatou component of a meromorphic function
$f$. Generally for any two distinct points $a$ and $b$ in $U$ there
exist $M>1$ and a positive integer $N$ such that for $\forall\ n\geq
N$,
\begin{equation}\label{equ1}|f^n(b)|^{1/M}\leq|f^n(a)|\leq |f^n(b)|^{M}\end{equation}
and if $\hat{\mathbb{C}}\setminus U$ contains an unbounded
component, we have a more precise inequality
\begin{equation}\label{equ1.4}M^{-1}|f^n(a)|\leq |f^n(b)|\leq M|f^n(a)|,\ \forall\ n\geq N.\end{equation}
(see \cite{Baker}, Lemma 5 or \cite{ber}, Lemma 7). Therefore, if
$U$ is an unbounded or simply-connected wandering domain, then
(\ref{equ1.4}) holds. As we know, an escaping Baker domain is
unbounded and may be either simply-connected or multiply-connected.
However, it was shown in \cite{Zheng4} and \cite{Rippon} that the inequality (\ref{equ1.4}) holds
for Baker domains. Therefore,
an escaping Fatou component $U$ that does not satisfy (\ref{equ1.4}) must
be bounded, multiply-connected and wandering.

The first entire function with multiply connected Fatou component
was constructed by I. N. Baker \cite{Baker1, Baker2, Baker3}.
Every multiply-connected Fatou component $U$ of a transcendental
entire function is a bounded escaping wandering domain  \cite{T,Baker4} and for all sufficiently large $n$,
$f^n(U)$ contains a round annulus centered at the origin with the modulus tending to $\infty$ as $n\to\infty$
\cite{Zheng2,BRS}. However, this is not
the case for meromorphic functions with poles. A meromorphic
function may have Herman rings, multiply connected attracting
domains, parabolic domains or Baker domains; see \cite{D} for
examples of meromorphic functions with only one pole that have an
invariant multiply connected Fatou component. A meromorphic function
was constructed by Baker, Kotus and L\"u \cite{BKL1} to have a
multiply connected wandering domain $U$ of preassigned connectivity
such that the limit set of $\{f^n|_U\}$ is an infinite set including
$\infty$ (a planar region $E$ is said to have connectivity $m$ if
$\hat{\mathbb{C}}\setminus E$ has $m$ components). Such a wandering
domain of a transcendental entire function must be simply connected
and a corresponding example has been constructed by Eremenko
and Lyubich \cite{EL}. Examples of meromorphic functions with
finitely or infinitely many poles, which have simply, doubly or
infinitely connected wandering domains and which have connectivity-changing wandering domains under iterates, can be found in
\cite{RipponStallard2}.

There exists a meromorphic function with bounded,
multiply-connected, escaping and wandering domains separating the
origin and $\infty$ where (\ref{equ1.4}) holds (see Example 5.2 of
\cite{Zheng8}). As we know, there also exist escaping wandering
domains such that (\ref{equ1}) holds while (\ref{equ1.4}) does not
hold. In fact, an entire function does not satisfy (\ref{equ1.4}) in
its multiply connected Fatou components; however, (\ref{equ1}) is
the best possible (see Theorem 1.1 in \cite{BRS}). Furthermore, if
(\ref{equ1.4}) does not hold, then $U$ contains two points $a$
and $b$ such that
\begin{equation}\label{equ2}|f^{n_k}(a)|/|f^{n_k}(b)|\to\infty
\quad (k\to\infty).
\end{equation}
We have shown in \cite{Zheng8} that, when \eqref{equ2} holds, $\bigcup_{n=0}^\infty f^n(U)$ contains a sequence of
annuli $A(r_m,R_m)$ with $R_m/r_m\to\infty (m\to\infty)$ and
$r_m\to\infty (m\to\infty)$; we call such a sequence of annuli {\it an
infinite modulus annulus sequence}. Here and henceforth, for $R>r>0$
we denote the round annulus $\{z:\ r<|z|<R\}$ by $A(r,R)$. In addition, for all
sufficiently large $n_k$, $U_{n_k}=f^{n_k}(U)$ separates the origin
and $\infty$. A condition was given in \cite{Zheng8} to ensure that for all sufficiently large $n$,
$U_{n}=f^{n}(U)$ surrounds the origin and contains a large round
annulus centered at the origin which forms an infinite modulus
annulus sequence. Also, we can construct meromorphic functions which have
a wandering domain $U$ such that $f^{2k}(U)$ surrounds the origin and contains a large round annulus centered at the origin, and \eqref{equ2} holds for $n_k = 2k$, but $f^{2k-1}(U)$ does not surround the origin, by
suitably modifying the construction of Example 5.4 in \cite{Zheng8}.
Here we omit the details which are routine but lengthy.

This leads us to  the following:

{\sl {\bf Question $\mathcal{A}$:}\ For an escaping Fatou component
$U$ of $f$, if $\bigcup_{n=0}^\infty f^n(U)$ contains an infinite
modulus annulus sequence, is there some pair of points $a$ and $b$
in $U$ such that (\ref{equ1.4}) does not hold? }

An equivalent statement of Question $\mathcal{A}$ is that if
(\ref{equ1.4}) holds in an escaping Fatou component $U$ of
$f$, then whether  $\bigcup_{n=0}^\infty f^n(U)$ may contain an
infinite modulus annulus sequence? It turns out that there exist wandering domains whose orbit contains an infinite
modulus annulus sequence and (\ref{equ1.4}) holds. This is the first
result of this paper.

\begin{theorem}\label{thm1}There exists a transcendental meromorphic function $f$
which has an escaping wandering domain $U$ such that for all
sufficiently large $n$, $f^n(U)\supset A(r_n,R_n)$ with
$R_n/r_n\to\infty (n\to\infty)$ and $r_n\to\infty (n\to\infty)$, but
for any two points $a, b\in U$, there exists an $M>1$ such that for
all sufficiently large $n$,
$$M^{-1}|f^n(a)|\leq |f^n(b)|\leq M|f^n(a)|.$$
\end{theorem}

However, the answer to Question $\mathcal{A}$ is positive for
transcendental entire functions. Indeed, (\ref{equ1.4}) does not hold
for an entire function in its multiply-connected Fatou components, as
we mentioned earlier (see Theorem 1.1 in \cite{BRS}). In this paper,
we generalize this result to meromorphic functions with infinitely
many poles.

In order to clearly state our second result of this paper, we
introduce some notations. From now on, we let
$M(r,f)$ ($\hat{m}(r,f)$, respectively) denote the maximal
(minimal, respectively) modulus of $f$ on the circle $\{|z|=r\}$,
i.e.,
$$M(r,f)=\max\{|f(z)|:\ |z|=r\},\quad \hat{m}(r,f)=\min\{|f(z)|:\ |z|=r\}.$$
If $f$ has a pole on $\{|z|=r\}$, then $M(r,f)=+\infty$. Then $T(r,f)$ refers to
the Nevanlinna characteristic function of $f$; $m(r,f)$ denotes the proximity
function of $f$; $N(r,f)$ is the integrated counting function of $f$
in Nevanlinna theory (their definitions are presented in Section 2; see \cite{Hayman1} and \cite{Zheng1}).

For an escaping Fatou component $U$ and $z_0 \in U$, define
$$h_n(z):=\frac{\log|f^n(z)|}{\log|f^n(z_0)}|,\ \forall\ z\in U,$$
$$\overline{h}(z):=\limsup_{n\to\infty} h_n(z),\ \underline{h}(z):=\liminf_{n\to\infty} h_n(z),\ \forall\ z\in
U.$$ In view of (\ref{equ1}) we easily see that $\overline{h}(z)$
and $\underline{h}(z)$ exist on $U$. In \cite{BRS}, Bergweiler,
Rippon and Stallard first introduced and proved that
$h(z):=\overline{h}(z)=\underline{h}(z)$ on a multiply connected
Fatou component $U$ of an entire function $f$ and they applied $h$
successfully to describe the geometric structure of $U$. This shows that $h$ is a significant function.

\begin{theorem}\label{thm6}\ Assume that
for some $C\geq 1$, $c>0$ and $r_0>0$, we have
\begin{equation}\label{equ1.5}\log M(Cr,f)\geq\left(1-\frac{c}{\log r}\right)T(r,f),\ \forall\ r>r_0.\end{equation}
Then there exist constants $D>1$, $d>1$ and $R_0>0$, which depend on $C,c$ and $r_0$, such that for a
Fatou component $U$ of $f$, if for some $m>0$ and some $r\geq R_0$,
we have
\begin{equation}\label{equ1.5+}A(D^{-d}r,D^dr)\subseteq
f^m(U),\end{equation} then $U$ is an escaping wandering domain,
$h(z):=\overline{h}(z)=\underline{h}(z)$ exists on $U$ and $h(z)$ is
a non-constant positive harmonic function on $U$. Therefore for any
point $z_0\in U$ and any neighborhood $V$ of $z_0$ in $U$, we have
for some constant $0<\alpha<1$ and all sufficiently large $n$
$$f^n(V)\supset A(|f^n(z_0)|^{1-\alpha},|f^n(z_0)|^{1+\alpha}).$$
Moreover, there exist two points $a$ and $b$ in $V$ and $\tau>1$
such that
$$|f^n(a)|^\tau\leq |f^n(b)|$$
and
$$A_n=A(r_n,R_n)\subseteq  f^n(V)$$ with
$R_n\geq r_n^\tau\to\infty \, (n\to\infty)$ and $A_{n+1}\subseteq
f(A_n)$.
\end{theorem}

We can give effective conditions to determine $D, d$ and $R_0$ in
Theorem \ref{thm6}. We also note that the conditions that $D $, $ d$ and
$R_0$ satisfy are chosen to guarantee the conclusion of Lemma \ref{lem6} to
hold. Therefore from the proof of Theorem 1.1 in \cite{Zheng8}, a
rough calculation shows that $D, d$ and $R_0$ satisfying the
following inequalities are sufficient for Theorem \ref{thm6}:
\begin{equation}\label{equ1.6}D>2Ce^{c+1},\
\frac{\pi}{d\cos\frac{\pi}{2d}}\left(1+\frac{\pi}{2\log
D}\right)<\frac{1}{6},\end{equation} $\ R_0\geq D^dr_0$ and
\begin{equation}\label{equ1.7+}\frac{T(r,f)}{\log r}>d\log D,\ \forall\ r\geq R_0.\end{equation}
For instance,  we have (\ref{equ1.6}) by taking $C=c=1$, $D=e^3$ and $d=10\pi$.

Since for a transcendental meromorphic function $f$,
$$\frac{T(r,f)}{\log r}\to\infty\ (r\to\infty),$$
it follows that if $\bigcup_{n=0}^\infty f^n(U)$
contains an infinite modulus annulus sequence for an escaping Fatou
component $U$ of $f$, then we have (\ref{equ1.5+}), where $D, d$ and
$R_0$ are chosen such that (\ref{equ1.6}) and (\ref{equ1.7+}) hold.

We give some remarks on (\ref{equ1.5}). Theorem \ref{thm1} tells us
that the condition (\ref{equ1.5}) cannot be removed in Theorem \ref{thm6}. In
Example 5.3 of \cite{Zheng8}, we constructed a meromorphic function
which has an escaping wandering domain $U$ in which there exist two
points $a$ and $b$ such that $|f^n(b)|/|f^n(a)|\to\infty \,
(n\to\infty)$ but $h_U(z)\equiv 1$ on $U$.

If $f$ has at most finitely many poles, then (\ref{equ1.5}) holds immediately. Hence, Theorem \ref{thm6} extends the related results
in \cite{BRS} to meromorphic functions with infinitely many poles.
If
\begin{equation}\label{equ2+}T(r,f)\geq N(r,f)\log r,\ \forall\ r\geq
r_0,\end{equation} then we have
\begin{equation*}m(r,f)=T(r,f)-N(r,f)\geq \left(1-\frac{1}{\log r}\right)T(r,f),\ \forall\ r\geq
r_0.
\end{equation*}
Also, (\ref{equ2+}) implies (\ref{equ1.5}) and that the Nevanlinna
deficiency $\delta(\infty,f)$ of $f$ of poles is $1$. It is easy to
find a meromorphic function $f$ satisfying (\ref{equ1.5}) and
$\delta(\infty,f)<1$. Consider the function
$f(z)=\dfrac{e^z-1}{e^{-z}+1}$. Then
$$m(r,f)\sim \frac{r}{\pi},\ \ N(r,f)\sim \frac{r}{\pi}\
(r\to\infty),\ \ \delta(\infty,f)=\frac{1}{2}$$ and
$$\log M(r,f)\geq \log\frac{e^r-1}{e^{-r}+1}\sim r\sim \frac{\pi}{2} T(r,f)\ (r\to\infty).$$
Then we obtain (\ref{equ1.5}) for $C=1$.

Finally, we mention that Theorem \ref{thm6} is still true even if
(\ref{equ1.5}) is satisfied for all $r$ possibly outside a set
with finite logarithmic measure. Indeed, for a set $E$ with finite logarithmic measure, for any $d>1$, we have $(r,dr)\setminus E\not=\emptyset$ for all sufficiently large $r$. Then in the proof of Theorem \ref{thm6}, we can choose available $r$ outside $E$.

The existence of large annuli in the orbit of an  escaping Fatou
component guarantees  the stability of the annuli in the sense that
the suitable large annuli still exist under relatively small changes
to the function in question.

\begin{theorem}\label{thm3}\ Let $f$ be a meromorphic function satisfying the conditions in Theorem \ref{thm6}.
If $f$ has an escaping Fatou component $U$ such that
$\bigcup_{n=0}^\infty f^n(U)$ contains an infinite modulus annulus
sequence, then for any meromorphic function $g$ which is analytic on
$\bigcup_{n=N}^\infty f^n(U)$ for some $N$ and satisfies
\begin{equation}\label{equ1.6+}M(r,f-g)\leq M(r,f)^\delta, \text{for}\ \{z:|z|=r\}\subset
\bigcup_{n=N}^\infty f^n(U),  \ \ \delta\in (0,1),
\end{equation} $g$ has an escaping Fatou
component $V$ on which the results in Theorem \ref{thm6} also hold for $g$.
\end{theorem}

Theorem \ref{thm3} was proved in \cite{BRS} for the case where
 $f$
and $g$ are entire. Theorem \ref{thm3} gives us a strong motivation to find a meromorphic function $f$ with $\delta(\infty,f)=0$
which has an escaping wandering domain $U$ satisfying the properties stated in
Theorem \ref{thm6}. This is stated as follows.

\begin{theorem}\label{thm5} There exists a meromorphic function $f$ with $\delta(\infty,f)=0$ which has a
wandering domain $W$ such that the results in Theorem \ref{thm6} hold for $f$
and $W$.
\end{theorem}

Furthermore, in terms of Theorem \ref{thm3}, we can find a meromorphic function to which Theorem 2 applies, but it does not satisfy (\ref{equ1.5}) in a set of
$r$ with infinite logarithmic measure. Indeed, we can make a modification of
$f$ in $\mathbb{C}\setminus \bigcup_{n=N}^\infty f^n(U)$ such that (\ref{equ1.5}) does not hold.

The organization of this paper is as follows. In Section 2, we collect some
preliminary results that will be needed in the proofs of theorems. The proofs of Theorems \ref{thm1}-\ref{thm5} are given in Section 3. We will make further
discussions or remarks after the proofs of theorems.

\section{Preliminary Lemmas}

Firstly we need a covering lemma which comes from the hyperbolic
metric. Let $\Omega$ be a hyperbolic domain in the extended complex
plane $\hat{\mathbb{C}}$, that is, $\hat{\mathbb{C}}\backslash
\Omega$ contains at least three points. Then there exists the hyperbolic metric $\lambda_{\Omega}(z)|{\rm d}z|$ with
Gaussian curvature $-1$ on
$\Omega$, where $\lambda_{\Omega}(z)$ is the hyperbolic
density of $\Omega$ at $z\in \Omega.$ In particular, we have
$$\lambda_{D}(z)=\frac{1}{|z|\log |z|},\ D=\{z:|z|>1\}$$
and  \cite[p.~791]{Zheng8}
\begin{equation}\label{2.1+}
\lambda_{A}(z)=\frac{\pi}{\left|z\right|{\rm mod}(A)\sin(\pi
\log(R/\left|z\right|)/{\rm mod}(A))}, \quad \forall\ z\in A=A(r,R),
\end{equation}
where $A(r,R)=\{z:\ r<|z|<R\}$ and ${\rm mod}(A)=\log R/r$ is the
modulus of $A$. The hyperbolic distance $d_{\Omega}(u,v)$ of two
points $u, v\in\Omega$ is defined by
\begin{equation*}\label{2.1}
d_{\Omega}(u,v)=\inf_{\alpha}\int_{\alpha}\lambda_{\Omega}(z)\left|dz\right|,
\end{equation*}
where the infimum is taken over all piecewise smooth paths $\alpha$
in $\Omega$ joining $u$ and $v$.

\begin{lemma}\label{lem2.2}(\cite{Zheng8}) \ Let $f$ be analytic on a hyperbolic domain $U$ with $0\not\in
f(U)$. If there exist two distinct points $z_1$ and $z_2$ in $U$
such that $|f(z_1)|> e^{\kappa\delta}|f(z_2)|$, where
$\delta=d_U(z_1,z_2)$ and $\kappa=\Gamma(1/4)^4/(4\pi^2)  \approx 4.37688 $, then
there exists a point $\hat{z}\in U$ such that $|f(z_2)|\leq
|f(\hat{z})|\leq |f(z_1)|$ and
\begin{equation}\label{2.1}f(U)\supset
A\left(e^\kappa\left(\frac{|f(z_2)|}{|f(z_1)|}\right)^{1/\delta}|f(\hat{z})|,\
e^{-\kappa}\left(\frac{|f(z_1)|}{|f(z_2)|}\right)^{1/\delta}|f(\hat{z})|\right).\end{equation}
If $|f(z_1)|\geq
\exp\left(\frac{\kappa\delta}{1-\delta}\right)|f(z_2)|$ and
$0<\delta<1$, then
\begin{equation}\label{2.2}
f(U)\supset A(|f(z_2)|,|f(z_1)|).
\end{equation}
In particular, for $\delta\leq\frac{1}{6}$ and $|f(z_1)|\geq
e|f(z_2)|$, we have (\ref{2.2}).
\end{lemma}


In order to cover an effective annulus, we are forced to
calculate carefully the hyperbolic distance in an annulus.

\begin{lemma}\label{lem3}\ Set $A=A(r,R)$ with $0<r<R$.
Then for any two points $z_1, z_2\in A$ with $|z_2|\leq |z_1|$ we
have
$$\max\left\{\frac{\pi}{{\rm mod}(A)}\log\frac{|z_1|}{|z_2|},
\log \frac{\log\frac{R}{|z_2|}}{\log\frac{R}{|z_1|}}\right\}\leq
d_A(z_1,z_2)\leq
 \frac{\pi^2}{{\rm
mod}(A)\hat{K}_0}+\frac{\pi}{{\rm
mod}(A)\hat{K}_1}\log\frac{|z_1|}{|z_2|},$$ where
\begin{eqnarray*}
\hat{K}_0=\max\limits_{j=1,2}\sin(\pi \log(R/|z_j|)/{\rm mod}(A)), \quad \hat{K}_1=\min\limits_{j=1,2}\sin(\pi \log(R/|z_j|)/{\rm
mod}(A)).
\end{eqnarray*}
 In particular, for $z_1, z_2$ with
$R/|z_1|=\left(R/r\right)^\sigma,
R/|z_2|=\left(R/r\right)^\tau$ and
$0<\sigma\leq\tau<1$, we have
\begin{equation}\label{2.3}\max\left\{(\tau-\sigma)\pi,\log\frac{\tau}{\sigma}\right\}\leq
d_A(z_1,z_2)\leq \frac{\pi^2}{\hat{K}_0{\rm
mod}(A)}+\frac{(\tau-\sigma)\pi}{\hat{K}_1},
\end{equation} with
$\hat{K}_0=\max\{\sin(\sigma\pi),\sin(\tau\pi)\}$ and
$\hat{K}_1=\min\{\sin(\sigma\pi),\sin(\tau\pi)\}$.
\end{lemma}

\begin{proof}\ We prove the first inequality. In view of (\ref{2.1+}), for all $z\in A$, we have
\begin{equation*}
\lambda_{A}(z)=\frac{\pi}{|z|{\rm mod}(A)\sin(\pi \log(R/|z|)/{\rm
mod}(A))}\geq \frac{1}{|z|\log\frac{R}{|z|}}
\end{equation*}
and
$$\lambda_{A}(z)\geq\frac{\pi}{|z|{\rm mod}(A)}.$$
Therefore
$$d_A(z_1,z_2)\geq \int_\gamma\frac{|{\rm d}z|}{|z|\log\frac{R}{|z|}}\geq \log
\frac{\log\frac{R}{|z_2|}}{\log\frac{R}{|z_1|}}$$ and
$$d_A(z_1,z_2)\geq \frac{\pi}{{\rm mod}(A)}\log\frac{|z_1|}{|z_2|},$$
 where $\gamma$
is the geodesic curve in $A$ connecting $z_1$ and $z_2$.

Next we prove the second inequality. For all $z$ with $|z_2|\leq
|z|\leq |z_1|$ we have
$$\sin(\pi \log(R/|z|)/{\rm
mod}(A))\geq\min_{j=1,2}\sin(\pi \log(R/|z_j|)/{\rm
mod}(A))=\hat{K}_1.$$ Without loss of generality, assume that
$\hat{K}_0$ attains the maximal value at $z_2$, as the same argument applies to the case when $z_2$ is replaced by $z_1$. Then
\begin{eqnarray*}d_A(z_1,z_2)&\leq& \frac{\pi}{{\rm
mod}(A)\hat{K}_0}\int_{\gamma_0}\frac{|{\rm d}z|}{|z|}+
\frac{\pi}{{\rm
mod}(A)\hat{K}_1}\int_{\gamma_1}\frac{|{\rm d}z|}{|z|} \\
&\leq& \frac{\pi^2}{{\rm mod}(A)\hat{K}_0}+\frac{\pi}{{\rm
mod}(A)\hat{K}_1}\log\frac{|z_1|}{|z_2|},\end{eqnarray*} where
$\gamma_0$ is the shortest arc from $|z_2|e^{i\arg z_1}$ to $z_2$
and $\gamma_1=\{z=te^{i\arg z_1}:\ |z_1|\leq t\leq |z_2|\}$.
\end{proof}
Note that the bounds of $d_A(z_1,z_2)$ depend only on
${\rm mod}(A)$, but otherwise are independent of the size of $R$ and $r$.

The next lemma is proved in \cite{Zheng8} using the density-decreasing property of the hyperbolic metric.
\begin{lemma}[\cite{Zheng8}] \label{lemxx}\ Let $h(z)$ be analytic on the annulus
$B=A(r, R)$ with $0<r<R<+\infty$ such that $|h(z)|>1$ on $B$. Then
\begin{eqnarray}\label{2.7}
\log\hat{m}(\rho,h)&\geq&
\exp\left(-\frac{\pi^2}{2}\max\left\{\frac{1}{\log\frac{R}{\rho}},\frac{1}{\log\frac{\rho}{r}}\right\}\right)\log
M(\rho,h)\nonumber\\
&\geq&\min\left\{\frac{\log\frac{\rho}{r}-\pi}{\log\frac{\rho}{r}+\pi},
\frac{\log\frac{R}{\rho}-\pi}{\log\frac{R}{\rho}+\pi}\right\}\log
M(\rho, h),\end{eqnarray} where $\rho\in(r,R)$ and
$\hat{m}(\rho,h)=\min\{|h(z)|: |z|=\rho\}.$
\end{lemma}

Secondly we need some basic notations and results from Nevanlinna
theory of meromorphic functions \cite{Hayman1,Zheng1}.
Set $\log^+x=\log\max\{1,x\}.$ Let $f$ be a meromorphic function.
Define
\begin{eqnarray*}
m(r,f)&:=&\frac{1}{2\pi}\int_0^{2\pi}\log^+|f(re^{i\theta})|{\rm d}\theta,
\\
N(r,f)&:=&\int_0^r\frac{n(t,f)-n(0,f)}{t}{\rm d}t+n(0,f)\log r,
\end{eqnarray*}
where $n(t,f)$ denotes the number of poles of $f$ counted according
to their multiplicities in $\{z: |z|<t\}$; sometimes we write
$n(t,\infty)$ for $n(t,f)$ and $n(t,0)$ for
$n\left(t,\frac{1}{f}\right)$ when $f$ is clear in the context, and
the Nevanlinna characteristic of $f$ by
$$T(r,f):=m(r,f)+N(r,f).$$
Then the deficiency of poles in terms of the Nevanlinna characteristic of $f$ is given by
$$\delta(\infty,f)=\liminf_{r\to\infty}\frac{m(r,f)}{T(r,f)}=1-\limsup_{r\to\infty}\frac{N(r,f)}{T(r,f)}.$$

In the following lemma we summarize some results on the Nevanlinna characteristic that will be used later; this lemma also holds for $\log M(r,f)$ when $f$ is
transcendental entire (see Theorem 2.2 of \cite{BRS}).

\begin{lemma}\label{lem4}\ Let $f$ be a transcendental meromorphic
function. Then

(1)\ $\frac{T(r,f)}{\log r}\rightarrow\infty (r\rightarrow\infty)$;

(2)\ There exists $r_0>0$ such that for $r_0\leq r<R$,
\begin{equation}\label{2.8}T(R,f)\geq\frac{\log R}{\log
r}T(r,f);\end{equation}

\end{lemma}

 Since $T(r,f)$ is a logarithmic convex function, the result (2) in Lemma
\ref{lem4} follows from (1) in Lemma \ref{lem4} of \cite{Zheng8}.

\section{Proofs of Theorems \ref{thm1}-\ref{thm5}}

To prove Theorem \ref{thm1}, we first recall Runge Theorem.

\

{\sl {\bf Runge Theorem}(cf. \cite{Rudin}).\ Let $W$ be a compact
set on the complex plane and let $f(z)$ be analytic on $W$. Assume
that $E$ is a set which intersects every component of
$\mathbb{C}\setminus W$. Then for any $\varepsilon>0$, there exists
a rational function $R(z)$ such that all poles of $R(z)$ lie in $E$
and
$$|f(z)-R(z)|<\varepsilon,\ \forall z\in W.$$}

Now we proceed to the proof of Theorem \ref{thm1}.

\

Take four sequences $\{r_n\},\ \{R_n\},\ \{r'_n\}$ and $\{R'_n\}$
such that $10<r_n<r_n'<R_n'<R_n$, $2\leq r'_n/r_n\leq 3$, $2\leq
R_n/R'_n\leq 3$, $9R_n<r_{n+1}$ and $R_n/r_n=R_{n+1}'/r_{n+1}'$.
Thus
\begin{eqnarray*}
\frac{R_n}{r_n}=\frac{R_{n+1}'}{r_{n+1}'}\leq\frac{1}{4}\frac{R_{n+1}}{r_{n+1}}, \quad \frac{R_n}{r_n}\geq 4^{n-1}\frac{R_1}{r_1}\to\infty \quad
(n\to\infty).
\end{eqnarray*}
Define
$$T_n(z)=\frac{r'_{n+1}}{r_n}z:\ A(r_n,R_n)\to
A(r'_{n+1},R'_{n+1}).$$ Take $a_n$ and $b_n$ such that $b_n>R_n+n$,
$b_n/R_n\to 1 (n\to\infty)$ and $a_n<r_{n}-n$,
$a_n/r_n\to 1\ (n\to\infty)$. Set $A_n=A(r_n,R_n),\
C_n=\{z:|z|=b_n\ \text{or}\ a_n\}$ and $B_n=B(0,r_n/4)$ with
$r_1>4$.

Choose a sequence of positive numbers $\{\varepsilon_n\}$ such that
$\varepsilon_{n+1}<\frac{1}{2}\varepsilon_n$ and
$\varepsilon_1<\frac{1}{2}$. In view of the Runge's Theorem, we have
a rational function $f_1(z)$ such that
$$|f_1(z)-T_1(z)|<\varepsilon_1, \forall\ z\in A_1;\ |f_1(z)|<\varepsilon_1, \quad \forall\ z\in
B_1$$
$$|f_1(z)|<\varepsilon_1, \quad \forall\ z\in C_1$$
and inductively, we have rational function $f_{n+1}(z)$ such that
$$\left|\sum_{k=1}^{n+1}f_k(z)-T_{n+1}(z)\right|<\varepsilon_{n+1},
\forall\ z\in A_{n+1};\ |f_{n+1}(z)|<\varepsilon_{n+1}, \quad \forall\
z\in B_{n+1}$$ and
$$\left|\sum_{k=1}^{n+1}f_k(z)\right|<\varepsilon_{n+1}, \quad \forall\
z\in C_{n+1}.$$ Write $f(z)=\sum_{n=1}^\infty f_n(z)$. Since this
series is uniformly convergent on any compact subset of
$\mathbb{C}$, $f(z)$ is a meromorphic function on $\mathbb{C}$.

For $z\in B_1$, we have
$$|f(z)|\leq\sum_{n=1}^\infty|f_n(z)|<\sum_{n=1}^\infty\varepsilon_n<1,$$
that is to say, $f(B_1)\subset B(0,1) \subset B_1$ and so $B_1$ is contained in an invariant Fatou component of $f$. For $z\in C_n$, we have, by
noting that $C_n\subset B_m$ for $m>n$,
$$|f(z)|\leq \left|\sum_{k=1}^nf_k(z)\right|+\sum_{k=n+1}^\infty|f_k(z)|
<\varepsilon_n+\sum_{k=n+1}^\infty|f_k(z)|<\sum_{k=n}^\infty\varepsilon_k<1$$
and so $f(C_n)\subset B(0,1)$ and $C_n$ is contained in a
preperiodic Fatou component of $f$. Since for $z\in
A_n=A(r_n,R_n)\subset B_{n+1}$,
\begin{equation}\label{equ3}|f(z)-T_n(z)|\leq
\left|\sum_{k=1}^nf_k(z)-T_n(z)\right|+\left|\sum_{k=n+1}^\infty
f_k(z)\right|<\sum_{k=n}^\infty \varepsilon_k=\varepsilon'_n\
(say),\end{equation} we have
\begin{equation}\label{equ3.2}f(A(r_n,R_n))\subset A\left(r'_{n+1}-\varepsilon'_n,\ R'_{n+1}+\varepsilon'_n\right)\subset
A(r_{n+1},R_{n+1}).\end{equation} Therefore, $A_n$ is contained in a
wandering domain $U_n$ of $f$ and $f: U_n\to U_{n+1}$ is proper.
Since $f$ is univalent in $A_n$ by Rouch\'e's Theorem, $U_n$ is
not doubly connected. Each of $U_n$ has no isolated boundary points.
If $U_n$ is finitely connected, then in view of the Riemann-Hurwitz
Theorem, $f$ is univalent in $U_n$, but the modulus of annulus $A_m$
tends to infinity as $m \rightarrow \infty$, which contradicts the conformal
invariance of annulus modulus. Therefore, $U_n$ must be
infinitely connected.

For a point $a\in A_n$, it follows from (\ref{equ3}) that
$$\frac{r'_{n+1}}{r_n}|a|-\varepsilon'_n\leq
|f(a)|\leq\frac{r'_{n+1}}{r_n}|a|+\varepsilon'_n$$ and
$$\frac{r'_{n+1}-\varepsilon'_n}{r_n}\leq
\frac{|f(a)|}{|a|}\leq\frac{r'_{n+1}+\varepsilon'_n}{r_n}.$$
Inductively, from (\ref{equ3.2}) we have
$$\prod_{k=1}^m\frac{r'_{n+k}-\varepsilon'_{n+k-1}}{r_{n+k-1}}\leq
\frac{|f^m(a)|}{|a|}\leq
\prod_{k=1}^m\frac{r'_{n+k}+\varepsilon'_{n+k-1}}{r_{n+k-1}}.$$ We
note that
\begin{eqnarray*}\prod_{k=1}^\infty\frac{r'_{n+k}+\varepsilon'_{n+k-1}}{r'_{n+k}-\varepsilon'_{n+k-1}}
&=&\prod_{k=1}^\infty\left(1+\frac{2\varepsilon'_{n+k-1}}{r'_{n+k}-\varepsilon'_{n+k-1}}\right)\\
&<&\exp\sum_{k=1}^\infty\frac{2\varepsilon'_{n+k-1}}{r'_{n+k}-\varepsilon'_{n+k-1}}<e^2.\end{eqnarray*}
Thus for two points $a$ and $b$ in $A_n$, we have
\begin{equation}\label{1++}\frac{|f^m(a)|}{|f^m(b)|}\leq\frac{|a|}{|b|}
\prod_{k=1}^m\frac{r'_{n+k}+\varepsilon'_{n+k-1}}{r'_{n+k}-\varepsilon'_{n+k-1}}<
\frac{|a|}{|b|}e^2, \quad \forall m\in \mathbb{N}.\end{equation}

Now we need to treat two cases:

Case (I):\ $|f^m(b)|\leq R'_m <R_m\leq |f^m(a)|$;

Case (II):\ $|f^m(b)|\leq r_m<r'_m\leq |f^m(a)|$.

Below we only prove that Case (I) would be impossible. The same
argument applies to Case (II) as well. Set $E_m=A(a_m,b_m).$
In view of the Principle of Hyperbolic Metric (Schwarz-Pick Lemma),
we can obtain that
\begin{eqnarray*}\ d_U(a,b)&\geq& d_{U_m}(f^m(a),f^m(b))\\
&\geq & d_{E_m}(f^m(b),f^m(a))\\
&\geq & d_{E_m}(R'_m,R_m)\\
&=&\int_{R'_m}^{R_m}\lambda_{E_m}(z)|{\rm
d}z|\\
&=&\int_{R'_m}^{R_m}\frac{\pi}{t\ {\rm
mod}(E_m)\sin(\pi\log\frac{b_m}{t}/{\rm mod}(E_m))}{\rm d}t\\
&\geq &\int_{R'_m}^{R_m}\frac{{\rm
d}t}{t\log\frac{b_m}{t}}\\
&=&\log\frac{\log\frac{b_m}{R'_m}}{\log\frac{b_m}{R_m}}\\
&=&\log\left(1+\frac{\log\frac{R_m}{R'_m}}{\log\frac{b_m}{R_m}}\right)\\
&\geq&
 \log\left(1+\frac{\log 2}{\log\frac{b_m}{R_m}}\right)\to\infty \quad
(m\to\infty).
\end{eqnarray*}
A contradiction is derived. Then for all sufficiently large $m>0$,
$f^m(a)$ and $f^m(b)$ lie in $A_m$, $A(a_m, r'_m)$  or
$A(R'_m,b_m)$ simultaneously.

By noting that $2\leq r'_n/a_n\leq 4$ and $2\leq b_n/R'_n\leq 4$
together with (\ref{1++}), we deduce that for any pair of points $a,
b\in U$ and all sufficiently large $n$, we have
$$\frac{|f^n(a)|}{|f^n(b)|}\leq\max\left\{4,\frac{|a|}{|b|}e^2\right\}.$$\qed

\

For the proof of Theorem \ref{thm6}, we first establish two lemmas.

\begin{lemma}\label{lem5}\ Let $f$ be a meromorphic function on $\{z:\ |z|\leq
R\}$ and $f$ be analytic on $\overline{A}(r,R)$ with $0<r<R/4$. Then
for $r+1<r'<R'\leq R$ and $|z|=r'$, we have
\begin{equation}\label{equ3.9}\log^+|f(z)|\leq\left(\frac{R'+r'}{R'-r'}+
\left(\log\frac{R}{r}\right)^{-1}\log \frac{R'}{r'-r}\right) T(R,f).
\end{equation}
\end{lemma}

\begin{proof}\  In view of the Poisson formula, for a point $z\in D=\{z:\ |z|<R'\}$ with $f(z)\not=0,\ \infty$, we have
$$\log|f(z)|=\frac{1}{2\pi}\int_{\partial
	D}\log|f(\zeta)|\frac{\partial G_D(\zeta,z)}{\partial \vec{n}}{\rm
	d}s-\sum\limits_{a_m\in D}G_D(a_m,z)+\sum\limits_{b_n\in D}G_D(b_n,z)$$		
so that 		
\begin{equation}\label{equ3.10}\ \log|f(z)|\leq m(D,z,f)+N(D,z,f),
\end{equation}
where
\begin{eqnarray*}
m(D,z,f)&=&\frac{1}{2\pi}\int_{\partial
D}\log^+|f(\zeta)|\frac{\partial G_D(\zeta,z)}{\partial \vec{n}}{\rm
d}s,
\\
N(D,z,f)&=&\sum\limits_{b_n\in D}G_D(b_n,z),
\end{eqnarray*}
$G_D$ is the Green function for $D$ and  all $a_m$ and $b_n$ are respectively zeros and  poles of $f$
in $D$ counted according to their multiplicities. Then for
$z$ with $|z|=r'$, we have
\begin{eqnarray*}
m(D,z,f)&=&\frac{1}{2\pi}\int_0^{2\pi}\log^+|f(R'e^{i\theta})|\frac{R'^2-|z|^2}{|R'^2-z|^2}{\rm
d}\theta
\\
&\leq&\frac{R'+|z|}{R'-|z|}m(R',f)
\\
&\leq&\frac{R'+|z|}{R'-|z|}T(R',f)
\end{eqnarray*}
and, by denoting the number of poles of $f$ in $D$ by $n(D,f)$, we
have \begin{eqnarray*}N(D,z,f)&\leq&\sum\limits_{b_n\in
D}\log\frac{R'}{|b_n-z|}\\
&\leq & n(D,f)\log \frac{R'}{r'-r}\\
&=&n(r,f)\log \frac{R'}{r'-r}\\
&=& \left(\log\frac{R}{r}\right)^{-1}\int_r^R\frac{n(t,f)}{t}{\rm
d}t\log \frac{R'}{r'-r}\\
&\leq& \left(\log\frac{R}{r}\right)^{-1}\left(\log
\frac{R'}{r'-r}\right) T(R,f).\end{eqnarray*} Therefore we
immediately have (\ref{equ3.9}).
\end{proof}

The following is extracted from the proof of Theorem 1.1 in
\cite{Zheng8}, but the condition (1.4) in that paper is replaced by
the more general condition (\ref{equ1.5}) in this paper.
The condition (1.4) in \cite{Zheng8} is that for arbitrarily large $C$
there exists an $s(C)>0$ such that for $r\geq s(C)$ we have
$$T(Cr,f)-N(Cr,f)\geq T(2r,f)+7\pi\log C.$$
Basically, in order to be able to use Lemma \ref{lem2.2} to obtain the large annulus in the proof of Theorem 1.1 in \cite{Zheng8}, we only deal with the following
$$T(2r,f)+7\pi\log C\leq m(Cr,f).$$
Thus we can take a point $z_1$ such that
$$\log|f(z_1)|\geq T(2r,f)+7\pi\log C.$$
Of course, we can also obtain $z_1$ required from the following inequality
$$T(2r,f)+7\pi\log C\leq \log M(Cr,f).$$
The inequality follows naturally from (\ref{equ1.5}) in this paper. In fact, under (\ref{equ1.5}), we have
\begin{eqnarray*}
\log M(2e^{c+2}Cr,f)&\geq &\left(1-\frac{c}{\log(2e^{c+2}Cr)}\right)T(2e^{c+2}r,f)\\
&\geq&\left(1-\frac{c}{\log(2r)}\right)\left(1+\frac{c+2}{\log(2r)}\right)
T(2r,f)\\
&>&T(2r,f)+\frac{T(2r,f)}{\log 2r}\\
&>&T(2r,f)+7\pi\log (2e^{c+2}C),
\end{eqnarray*} for large $r$. We can establish the following by replacing $C$ in the proof of Theorem 1.1 in \cite{Zheng8} with $D=2e^{c+2}C$.

\begin{lemma}\label{lem6}\ Let $f$ be a meromorphic function such that
the inequality (\ref{equ1.5}) holds. Let $U$ be a Fatou component of $f$. Then there
exist constants $D>1$, $d>1$ and $R_0>0$ such that if for some $m>0$
and $r\geq R_0$,
\begin{equation}\label{equ3.6}A(D^{-d}r,D^dr)\subseteq f^m(U),\end{equation}
then for all sufficiently large $n$, $f^n(U)\supset A_n$ and
$A_{n+1}\subset f(A_n)$ with $A_n=A(r_n,R_n)$, $R_n\geq r_n^\sigma$,
$r_{n+1}>R_n$, $r_n\to\infty\ (n\to\infty)$ and $\sigma>1$.
\end{lemma}

We remark that if the conclusion of Lemma \ref{lem6} holds, that is
to say, for large $n$, $f^n(U)$ contains the large annulus given by
Lemma \ref{lem6}, then the condition (\ref{equ1.5}) can be replaced
by the following weaker inequality
\begin{equation}\label{equ3.7}\log M(Cr,f)\geq\left(1-\frac{\log\log
r}{\log r}\right)T(r,f), \quad \forall\ r>r_0.
\end{equation}

We will complete the proof of Theorem \ref{thm6} under the condition
(\ref{equ3.7}) instead of (\ref{equ1.5}) after we obtain the large
annulus sequence given in Lemma \ref{lem6}. Now we are in the
position to prove Theorem \ref{thm6}.

\

Under the conditions of Theorem \ref{thm6}, in view of Lemma \ref{lem6}, for
all sufficiently large $n$, we have
\begin{equation}\label{equ3.8}A(r_n,R_n)\subset f^n(U)\end{equation} with
$R_n\geq r_n^\sigma$, $r_{n+1}>R_n$, $r_n\to\infty\ (n\to\infty)$
and $\sigma>1$. Therefore, for a sufficiently large $m$ and $n\geq
m$ we have $r_{n+1}>R_n\geq R_m^{\sigma^{n-m}}$. Without loss of
generality, for a sufficiently large $m$, we rewrite $r_m,\ R_m$ and
$\sigma>1$ such that
$$f^m(U)\supset A(r^\alpha_m,R^\beta_m)\supset A(r_m,R_m)$$
with $R_m=r^\sigma_m$, $(\sigma-1)\log r_m>2m^2$, $0<\alpha<1$ and
$\beta>1$ (for example, with $n=m$ in (\ref{equ3.8}),  we replace $r_m, R_m$
by $r_m^{1/\alpha}, R_m^{1/\beta}$ respectively and choose
$\beta=\alpha\sigma$). Take $R'_m=e^{-m^2}R_m$ and $r'_m=er_m$.
Applying Lemma \ref{lemxx} to $f$ on $A(r_m,R_m)$, we have
\begin{equation}\label{equ3.9+}\hat{m}(R'_m,f)\geq M(R'_m,f)^{s_m},\
s_m=\exp\left(-\frac{\pi^2}{2}\max\left\{\frac{1}{\log\frac{R_m}{R'_m}},
\frac{1}{\log\frac{R'_m}{r_m}}\right\}\right).\end{equation} Next we
estimate $s_m$. Since
$$\log\frac{R_m}{R'_m}=m^2\ \text{and}\ \log\frac{R'_m}{r_m}=(\sigma-1)\log
r_m-m^2>m^2,$$ we have
$$s_m=\exp\left(-\frac{\pi^2}{2m^2}\right).$$
Select a $\tau \in (1, \sigma)$ such that $s_m\sigma>\tau$ and take
two points $a$ and $b$ with $|a|=r'_m$ and $|b|=R'_m$. Set
$\hat{R}_m=r'_me^{m^3}$. Finding $\sigma_m$ such that
$R'_m/C={\hat{R}_m}^{\sigma_m},$ we have
$$\sigma_m=\frac{\log R'_m/C}{\log \hat{R}_m}=\sigma-\frac{\sigma m^3+m^2+\sigma+\log
C}{\log r_m+m^3+1}.$$

Therefore, we have $|f(a)|\leq M(r'_m,f)$ and in view of
(\ref{equ3.7}), (\ref{equ3.9+}) and Lemma \ref{lem4}, we have
$$|f(b)|\geq M(R'_m,f)^{s_m}\geq\exp (s_m\tau_m T(R'_m/C,f))$$
\begin{equation}\label{equ3.9++}= \exp(s_m\tau_mT({\hat{R}_m}^{\sigma_m},f))\geq
\exp(s_m\sigma_m\tau_mT(\hat{R}_m,f)),\end{equation} where
$\tau_m=1-\frac{\log\log (R'_m/C)}{\log (R'_m/C)}$. Since
$r'_m>r_m>r_0^{\sigma^m}$, where $r_0>e$, we have
$$\sigma_m>\sigma-\frac{(C+\sigma) m^3}{\log r_m}>\sigma-\frac{(C+\sigma)m^3}{\sigma^m\log r_0}$$ and
$s_m\sigma_m\tau_m\to\sigma>1 (m\to\infty)$.

In view of Lemma \ref{lem5} with $R'=r_m'(m^2+1)$, $r'=r'_m$,
$r=r_m$ and $R=\hat{R}_m$, we have
\begin{eqnarray*}\log M(r'_m,f)&\leq& \left(\frac{m^2+2}{m^2}
+\frac{\log\frac{r'_m(m^2+1)}{r'_m-r_m}}{\log\frac{\hat{R}_m}{r_m}}\right)T(\hat{R}_m,f)\\
&= &\left(\frac{m^2+2}{m^2}+
\frac{\log\frac{e}{e-1}(m^2+1)}{m^3+1}\right)T(\hat{R}_m,f)\\
&\leq &\left(1+\frac{t}{m^2}\right)T(\hat{R}_m,f),
\end{eqnarray*}
where $t$ is an absolute constant. Set
$t_m=\left(1+\frac{t}{m^2}\right)^{-1}$. Thus combining
(\ref{equ3.9++}) together with the above inequality yields that
\begin{equation}\label{equ3.10+}|f(b)|\geq \exp(s_m\sigma_m\tau_m T(\hat{R}_m,f)) \geq
M(r'_m,f)^{t_ms_m\sigma_m\tau_m}\geq
|f(a)|^{t_ms_m\sigma_m\tau_m}.\end{equation}

Set $r'_{m+1}=|f(a)|$ and $R'_{m+1}=|f(b)|$. Then we have
$R'_{m+1}\geq (r'_{m+1})^{t_ms_m\sigma_m\tau_m}$. Set
$R_{m+1}=e^{(m+1)^2}R'_{m+1}$ and $r_{m+1}=r'_{m+1}/e.$

Now we estimate $d_{U_m}(a,b)$ with $U_m=f^m(U)$ in terms of Lemma \ref{lem3}. Set
$A'_m=A(r_m^\alpha,R_m^\beta)$. Since
\begin{eqnarray*}
\sin\left(\pi\frac{\log\frac{R_m^\beta}{r'_m}}{{\rm
mod}(A'_m)}\right)&=&\sin\left(\pi\frac{(\beta-1/\sigma)\log
R_m-1}{(\beta-\alpha/\sigma)\log R_m}\right)
\\
&\to&\sin\left(\pi\frac{\sigma\beta-1}{\sigma\beta-\alpha}\right)
\quad (m\to\infty)
\end{eqnarray*}
and
\begin{eqnarray*}
\sin\left(\pi\frac{\log\frac{R_m^\beta}{R'_m}}{{\rm
mod}(A'_m)}\right)&=&\sin\left(\pi\frac{(\beta-1)\log
R_m+m^2}{(\beta-\alpha/\sigma)\log R_m}\right)
\\
&\to&\sin\left(\pi\frac{\sigma(\beta-1)}{\sigma\beta-\alpha}\right)
\quad (m\to\infty),
\end{eqnarray*}
we assume that $0<\hat{K}_1\leq \hat{K}_0<\min\left\{\sin\left(\pi\frac{\sigma\beta-1}{\sigma\beta-\alpha}\right),
\sin\left(\pi\frac{\sigma(\beta-1)}{\sigma\beta-\alpha}\right)\right\}$. In terms of Lemma \ref{lem3}, we have, for $|a|=r'_m$ and $|b|=R'_m$,
\begin{eqnarray*}\delta&=&d_{U_m}(a,b)\\
	&\leq& d_{A'_m}(a,b)\\
	&\leq&\frac{\pi^2}{{\rm mod}(A'_m)\hat{K}_0}+\frac{\pi}{{\rm mod}(A'_m)\hat{K}_1}\log\frac{R'_m}{r'_m}\\
		&=&\frac{\pi^2}{{\rm
				mod}(A'_m)\hat{K}_0}+\frac{\pi}{(\beta-\alpha/\sigma)\hat{K}_1\log R_m}((1-1/\sigma)\log
		R_m-m^2-1)\\
		& \to &\frac{\pi(\sigma-1)}{(\sigma\beta-\alpha)\hat{K}_1}
		\quad (m\to\infty).
\end{eqnarray*}
Therefore, for sufficiently large $m$ and $\sigma>1$ close to $1$,
we have $\delta=d_{U_m}(a,b)<\frac{1}{6}.$ In
view of the inequality, (\ref{equ3.10+}) and (\ref{equ3.9++}),
we have
\begin{eqnarray*}e^{-\kappa}\left(\frac{|f(b)|}{|f(a)|}\right)^{1/\delta-1}&\geq&
e^{-\kappa}\left(|f(b)|^{1-1/(t_ms_m\sigma_m\tau_m)}\right)^{5}\\
&\geq &
\exp\left[5\left(1-\frac{1}{t_ms_m\sigma_m\tau_m}\right)s_m\sigma_m\tau_m
T(\hat{R}_m,f)-\kappa\right]\\
&>&\exp[3(\sigma-1) T(\hat{R}_m,f)]
\\
&>&e^{(m+1)^2}.
\end{eqnarray*}
Thus, for $|f(a)|\leq |f(\hat{z})|\leq |f(b)|$ with $\hat{z}\in U_m$
we have
$$e^\kappa\left(\frac{|f(a)|}{|f(b)|}\right)^{1/\delta}|f(\hat{z})|\leq
e^\kappa\left(\frac{|f(a)|}{|f(b)|}\right)^{1/\delta-1}r'_{m+1}<r'_{m+1}/e=r_{m+1};$$
and
$$e^{-\kappa}\left(\frac{|f(b)|}{|f(a)|}\right)^{1/\delta}|f(\hat{z})|\geq
e^{-\kappa}\left(\frac{|f(b)|}{|f(a)|}\right)^{1/\delta-1}R'_{m+1} >
e^{(m+1)^2}R'_{m+1}=R_{m+1}.$$ Since $d_{U_m}(a,b)<\frac{1}{6}$, according to Lemma \ref{lem2.2}, we have
$$U_{m+1}=f(U_m)\supset A(r_{m+1},R_{m+1})\supset \overline{A}(r'_{m+1},R'_{m+1}).$$
Thus, by noting that $d_{U_{m+1}}(f(a),f(b))\leq
d_{U_m}(a,b)<\frac{1}{6}$ ($\beta$ and $\alpha$ are just used to
imply the inequality), $|f(a)|=r'_{m+1}$ and $|f(b)|=R'_{m+1}$ , we
can repeat the above step and obtain that
\begin{equation}\label{equ}|f^n(b)|\geq
|f^n(a)|^{\prod_{k=0}^{n-1}s_{k+m}t_{k+m}\sigma_{k+m}\tau_{k+m}}.\end{equation}
For sufficiently large $m$, we can require, for $n\geq 1$,
$$\prod_{k=0}^{n-1}s_{k+m}t_{k+m}\sigma_{k+m}\tau_{k+m}>\frac{\sigma+1}{2}>1.$$

Take a point $c\in U$. Define
$$h_n(z)=\frac{\log|f^n(z)|}{\log|f^n(c)|},\ \forall\ z\in U.$$
Since $U$ is wandering and escaping, we can assume that $|f^n(z)|>1$
on $U$ and hence with (\ref{equ1}) we conclude that $h_n(z)$ is harmonic and
positive on $U$. In view of (\ref{equ1}) or by Harnack's inequality,
the family $\{h_n\}$ of harmonic functions is locally normal on $U$
and hence
$$\overline{h}(z):=\limsup_{n\to\infty} h_n(z),\ \forall\ z\in U$$
exists and $\overline{h}$ is harmonic and positive on $U$. Then it follows from (\ref{equ}) that
$$h_n(b)\geq \frac{\sigma+1}{2}h_n(a), \quad
\overline{h}(b)\geq\frac{\sigma+1}{2}\overline{h}(a)>\overline{h}(a).$$
Therefore, $\overline{h}(z)$ is not a constant on $U$. The same
argument implies that $$\underline{h}(z):=\liminf_{n\to\infty}
h_n(z),\ \forall\ z\in U$$ exists and $\underline{h}$ is a
non-constant, harmonic and positive function on $U$.

Suppose that $\overline{h}(z_0)>\underline{h}(z_0)$ for some $z_0\in
U$. Without any loss of generality, suppose that
$\overline{h}(z_0)>\overline{h}(c)=1$. Take a real number $\eta$
with $\max\{1,\underline{h}(z_0)\}<\eta<\overline{h}(z_0).$ Since
$\overline{h}(z)$ is not a constant, we can find $z_1$ and $z_2$
such that
$\overline{h}(z_1)<1=\overline{h}(c)<\overline{h}(z_0)<\overline{h}(z_2)$.
In view of Lemma \ref{lem2.2}, for some sufficiently large $m$, we
have
$$f^m(U)\supset A\left(|f^m(c)|^\alpha,|f^m(c)|^\beta\right),$$
where $\overline{h}(z_1)<\alpha<1$ and
$1<\overline{h}(z_0)<\beta<\overline{h}(z_2).$

From the argument of the above paragraph we can take $a=f^m(c)$ and
$b=f^m(z_0)$ such that
$$|f^{n+m}(z_0)|>|f^{n+m}(c)|^\eta,\ \forall\ n\in \mathbb{N}.$$
This implies that $\underline{h}(z_0)\geq \eta$. A contradiction is
derived and so we have proved that $\overline{h}(z)=\underline{h}(z)$ on
$U$, that is to say,
$$h(z)=\lim_{n\to\infty}\frac{\log|f^n(z)|}{\log|f^n(c)|}$$
exists on $U$.

Set
$$H(z):=\lim_{n\to\infty}\frac{\log|f^n(z)|}{ \log|f^n(z_0)|}=\frac{h(z)}{h(z_0)}.$$
Then $H$ is a non-constant, harmonic and positive function on $U$. Thus
we can find two points $z_1$ and $z_2$ in $V$ and $0<\alpha<1$ such
that $H(z_1)<1-\alpha<1<1+\alpha<H(z_2)$. For a sufficiently large
$m$, we have $d_{f^m(V)}(f^m(z_1),f^m(z_2))<\frac{1}{6}$ and hence,
in view of Lemma \ref{lem2.2}, for $n\geq m$
$$f^n(V)\supset A(|f^n(z_1)|,|f^n(z_2)|)\supset
A(|f^n(z_0)|^{1-\alpha},|f^n(z_0)|^{1+\alpha}).$$

Thus we have (\ref{equ3.6}) for sufficiently large $m$ and $f^m(V)$
instead. In view of Lemma \ref{lem6}, for all sufficiently large $n$
and some $\tau>1$, we obtain $A_n=A(r_n,R_n)\subseteq  f^n(V)$ with
$R_n\geq r_n^\tau\to\infty (n\to\infty)$ such that $A_{n+1}\subset
f(A_n)$.

This completes the proof of Theorem \ref{thm6}. \qed

 \

 {\it Proof of Theorem \ref{thm3}.} In view of Theorem \ref{thm6}, there exist a sufficiently large $m$ and a sufficiently
 large $r_m$ such that for $n\geq m$ and $r_{n+1}>R_n\geq
 r_n^\sigma$ with $\sigma>1$ we have
 $$f^n(U)\supset A(r_n^\alpha, R_n^\beta),\ 0<\alpha<1,\ 1<\beta.$$
From (\ref{equ1.6+}) and (\ref{equ1.5}), by using Lemma
\ref{lemxx}, the version of Lemma \ref{lem4} for $M(r,f)$ from \cite{BRS} and Lemma \ref{lem5}, we can
show that
$$M(R_m,g)\geq \hat{m}(R_m,g)>M(r_m,g)\geq \hat{m}(r_m,g)$$
and in view of the maximum principle, we have
\begin{equation}\label{equ3.12}g(A(r_m,R_m))\subseteq A(\hat{m}(r_m,g), M(R_m,g)).\end{equation} Set
$u(z)=f(z)-g(z)$ and $s_m=\exp\left(-\frac{\pi^2}{2m^2}\right)$. For
sufficiently large $m$, we have $s_m>\max\{1/2,\delta\}$. Using
Lemma \ref{lemxx} with $r=e^{-m^2}r_m$, $\rho=r_m$ and $R=R_m$ and
noting that $\log(1-x)>-2x$ for $0<x<\frac{1}{2}$, we have
$$\hat{m}(r_m,g)\geq \hat{m}(r_m,f)-M(r_m,u)\geq
M(r_m,f)^{s_m}-M(r_m,u)$$
$$=\left(1-\frac{M(r_m,u)}{M(r_m,f)^{s_m}}\right)M(r_m,f)^{s_m}\geq
M(r_m,f)^{s_m\gamma_m},$$ where
$\gamma_m=1-\frac{4M(r_m,f)^{\delta-s_m}}{\log
M(r_m,f)}>1-\frac{1}{m^2}$,  and from the inequality $\log (1+x)<x$ for
$x>0$, we have $$M(R_m,g)\leq
M(R_m,f)+M(R_m,u)$$$$=M(R_m,f)\left(1+\frac{M(R_m,u)}{M(R_m,f)}\right)\leq
M(R_m,f)^{\tau_m},$$ where
$\tau_m=1+\frac{1}{M(R_m,f)^{1-\delta}\log
M(R_m,f)}<1+\frac{1}{m^2}$ for sufficiently large $m$. Thus it
follows from (\ref{equ3.12}) that
\begin{equation}\label{equ3.12+}g(A(r_m,R_m))\subseteq
A(M(r_m,f)^{s_m\gamma_m},M(R_m,f)^{\tau_m}).\end{equation}
 On the other hand, we have
 \begin{equation}\label{equ3.13}f^{m+1}(U)\supset f(A(r_m^\alpha, R_m^\beta))\supset
 A(M(r_m^\alpha,f),\hat{m}(R_m^\beta e^{-m^2},f)).\end{equation}
 By Lemma \ref{lemxx} and (\ref{equ1.5}), we have
 $$\hat{m}(R_m^\beta e^{-m^2},f)\geq M(R_m^\beta e^{-m^2},f)^{s_m}\geq \exp(\sigma_mt_m
 T(R_m^\beta e^{-m^2}/C,f)),$$
where $t_m=1-\frac{\log\log (R_m^\beta e^{-m^2}/C)}{\log
 (R_m^\beta e^{-m^2}/C)}>1-\frac{1}{m^2}$ for sufficiently large $m$,  and in view of Lemma \ref{lem5},
 we can show that
$$\tau_m\log M(R_m,f)<s_mt_m
 T(R_m^\beta e^{-m^2}/C,f)$$
 and by (\ref{equ1.5}) and  the version of Lemma \ref{lem4} for $M(r,f)$ from \cite{BRS}, we have
$$s_m\gamma_m \log M(r_m,f)>\log M(r_m^\alpha,f).$$
Thus $M(R_m^\beta e^{-m^2},f)^{s_m}>M(r_m^\alpha,f)$ and from
(\ref{equ3.13}) we can deduce
\begin{equation}\label{equ3.14}f^{m+1}(U)\supset
A(M(r_m^\alpha,f),M(R_m^\beta e^{-m^2},f)^{s_m}).\end{equation}

Now we set $\hat{r}_1=M(r_m^\alpha,f)^{s_m\gamma_m/\alpha}$ and
$\hat{R}_1=M(R_m,f)^{\tau_m}$ so that we have
$M(r_m,f)^{s_m\gamma_m}\geq \hat{r}_1$. Combining (\ref{equ3.12+})
and (\ref{equ3.14}) yields
$$f^{m+1}(U)\supset
A(\hat{r}_1^{\alpha\eta_m},\hat{R}_1^{\beta_m})\supset
A(\hat{r}_1,\hat{R}_1)\supset g(A(r_m,R_m))$$ with
$\eta_m=\frac{1}{s_m\gamma_m}>1$ and
$\beta_m=\frac{s_m}{\tau_m}\frac{\log M(R_m^\beta e^{-m^2},f)}{\log
M(R_m,f)}\to\beta\ (m\to\infty)$. For sufficiently large $m$, we can
require that $0<\alpha\eta_m<\frac{\alpha+1}{2}<1$ and
$1<\frac{\beta+1}{2}<\beta_m$.

 By the same argument as above, we have
$$f^{m+2}(U)\supset
A(\hat{r}_2^{\alpha\eta_m\eta_{m+1}},\hat{R}_2^{\beta_{m+1}})\supset
A(\hat{r}_2,\hat{R}_2)\supset g^2(A(r_m,R_m))$$ with
$\hat{r}_2=M(\hat{r}_1^{\alpha\eta_m},f)^{1/(\eta_{m+1}\eta_m\alpha)}$,
$\hat{R}_2=M(\hat{R}_1,f)^{\tau_{m+1}}$,
$0<\alpha\eta_m\eta_{m+1}<\frac{\alpha+1}{2}<1$ and
$1<\frac{\beta+1}{2}<\beta_{m+1}$.

Inductively, we have
$$f^{m+n}(U)\supset g^n(A(r_m,R_m)).$$
Thus $A(r_m,R_m)$ is contained in an escaping Fatou component $W$ of $g$.


We note that in $A(r_m,R_m)$ we have \begin{eqnarray*}\log
M(Cr,g)&\geq&\log(M(Cr,f)-M(Cr,u))\\
&= &\log M(Cr,f)+\log\left(1-\frac{M(Cr,u)}{M(Cr,f)}\right)\\
&\geq&\left(1-\frac{2\log\log r}{\log r}\right)T(r,f)\end{eqnarray*}
and for $r_m<r+1<r'<R'<R\leq R_m$ we have
\begin{eqnarray*}\log
M(r',g)&\leq &\log
M(r',f)+\log\left(1+\frac{M(r'u)}{M(r',f)}\right)\\
&\leq &\left(\frac{R'+r'}{R'-r'}+
\left(\log\frac{R}{r}\right)^{-1}\log \frac{R'}{r'-r}\right)
T(R,f)+1.
\end{eqnarray*}
Thus we can repeat the arguments in the proof of Theorem \ref{thm6} and
prove that the results of Theorem \ref{thm6} hold for $g$ in $W$. \qed

We make a remark on (\ref{equ1.6+}). From the proof of Theorem
\ref{thm3} we see that in order that Theorem \ref{thm3} holds for
$g$, in fact, we only need to require (\ref{equ1.6+}) holds in a
sequence of annuli $A(r_n,r_n^{\sigma_n})$ instead of $f^n(U)$ as
long as $A(r_n,r_n^{\sigma_n})\subset f^n(U)$ for $\forall\ n\geq m$
with $m$ being large enough and $f(A(r_n,r_n^{\sigma_n}))\subset
A(r_{n+1},r_{n+1}^{\sigma_{n+1}})$, where $1<\sigma\leq \sigma_n$.

\

{\it Proof of Theorem \ref{thm5}.}\  With the help of Theorem \ref{thm3} we
will find a meromorphic function $f$ with
$\delta(\infty,f)=0$ which has an escaping Fatou wandering domain
$U$ such that the results of Theorem \ref{thm6} hold for $f$ in $U$.

In \cite{BRS}, Bergweiler, Rippon and Stallard proved the results of
Theorem \ref{thm6} for entire functions in their multiply connected
Fatou components. The first entire function with multiply connected
Fatou component is due to I. N. Baker \cite{Baker1}. The multiply
connected wandering domains which have uniformly perfect boundary or
non-uniformly perfect boundary were found in \cite{BZ}. For example,
in Theorem 1.3 of \cite{BZ}, for a
sequence $\{r_k\}$ of positive numbers with $r_{k+1}>2r_k^2$ and a
sequence $\{\varepsilon_k\}$ of positive numbers tending to $0$ as
$k\to\infty$, an entire function $f$ of zero order is
constructed such that
$$f(A((1+\varepsilon_k)r_k,(1-\varepsilon_k)r_{k+1}))\subset
A((1+\varepsilon_{k+1})r_{k+1},(1-\varepsilon_{k+1})r_{k+2}).$$  Moreover, the Fatou components $U$ containing
$A((1+\varepsilon_k)r_k,(1-\varepsilon_k)r_{k+1})$ may have
uniformly perfect boundaries by choosing $\{r_k\}$ suitably.

Now we put
$$g(z)=f(z)+\sum_{n=1}^\infty\frac{1}{2^nm_n}\sum_{k=1}^{m_n}\frac{\varepsilon_n}{z-z_{k,n}},$$
where $m_n=[r_{n+1}]+1$, $[r_{n+1}]$ is the maximal integer not
exceeding $r_{n+1}$,  and $z_{k,n}=r_ne^{i\frac{2k\pi}{m_n}}$. It is
clear that $g$ is a meromorphic function. For $|z|=r$ with
$(1+\varepsilon_k)r_k<r<(1-\varepsilon_k)r_{k+1}$, we have
\begin{eqnarray*}|f(z)-g(z)|&\leq&\sum_{n=1}^\infty\frac{1}{2^nm_n}\sum_{j=1}^{m_n}\frac{\varepsilon_n}{|z-z_{j,n}|}\\
&\leq&\sum_{\begin{array}{c}n=1\\
n\not=k,k+1\end{array}}^\infty\frac{1}{2^n}\frac{\varepsilon_n}{|r-r_n|}\\
&\ &+\frac{1}{2^k}\frac{\varepsilon_k}{r-r_k}+
\frac{1}{2^{k+1}}\frac{\varepsilon_{k+1}}{r_{k+1}-r}\\
&<&\sum_{n=1}^\infty\frac{1}{2^n}=1,
\end{eqnarray*}
and this implies (\ref{equ1.6+}) for $|z|=r$ with
$(1+\varepsilon_k)r_k<r<(1-\varepsilon_k)r_{k+1}$. In view of
Theorem \ref{thm3} and from the remark after the proof of Theorem
\ref{thm3}, there exists an escaping wandering domain $W$ of $g$
such that the results of Theorem \ref{thm6} hold for $g$ and $W$.

Next we calculate the deficiency $\delta(\infty,g)$. For $er_k<r<(1-\varepsilon_k)r_{k+1}$, we have $m(r,g)\leq
m(r,f)+\log 2$ and
$$N(r,g)=\int_0^r\frac{n(t,g)}{t}{\rm d}t\geq \int_{r_k}^r\frac{m_k}{t}{\rm d}t\geq r_{k+1}\geq r.$$
Then as $r\in\cup_{k=m}^\infty
A(er_k,(1-\varepsilon_k)r_{k+1})\to\infty$, we have
$$\frac{m(r,g)}{T(r,g)}\leq\frac{m(r,g)}{N(r,g)}\leq\frac{m(r,f)+\log
2}{r}\to 0,$$ by using the fact that $f$ is of zero order. This implies that
$$\delta(\infty,g)=\liminf_{r\to\infty}\frac{m(r,g)}{T(r,g)}=0.$$
Then $g$ is the desired meromorphic function of Theorem \ref{thm5}.
\qed

Nevertheless we do not know if $W$ has uniformly perfect boundary. Therefore,
we ask the following question

\

{\sl {\bf Question $\mathcal{B}$:}\ Is there any meromorphic
function with zero Nevanlinna deficiency at poles which has a
multiply connected escaping wandering domain with uniformly perfect
boundary?}

\

Perhaps the following approach would be possible to solve Question
$\mathcal{B}$. We try to control the changes of critical values as
an entire function, that has a multiply connected escaping wandering
domain with uniformly perfect boundary, is changed to a
meromorphic function with zero Nevanlinna deficiency at poles and
then in view of Theorem \ref{thm3} we show the meromorphic function also has
a multiply connected escaping wandering domain with uniformly
perfect boundary.

\section*{Acknowledgements}
The authors would like to thank the anonymous referee for reading the manuscript
carefully and providing valuable suggestions which have improved the readability of this paper considerably.

\section*{Funding}
J.H. Zheng was supported by the National Natural Science Foundation of China (grant number 12071239).
C.F. Wu was supported by the National Natural Science Foundation of China (grant numbers 11701382 and 11971288); and Guangdong Basic and Applied Basic Research Foundation, China (grant number 2021A1515010054).

\end{document}